\documentclass{amsart}

\usepackage{amsaddr}

\usepackage{amsmath}
\usepackage{amsthm}
\usepackage{amssymb}

\usepackage{epsf,graphicx,graphics}

\usepackage{url}

\usepackage{xcolor}
\newcommand{\DM}[1]{{\small\color{violet}#1}}
\newcommand{\Q}[1]{{\color{red}#1}}

\newcommand{\comment}[1]{}

\usepackage{hyperref}
\hypersetup{bookmarks=true,
	unicode=true,
	colorlinks=true,
	citecolor=black,
	linkcolor=black,
	urlcolor=black,
	plainpages=false,
	pdfpagelabels=true}

\theoremstyle{plain}
\newtheorem{teo}{Theorem}[section]
\newtheorem{lemma}[teo]{Lemma}
\newtheorem{cor}[teo]{Corollary}
\newtheorem{prop}[teo]{Proposition}

\theoremstyle{definition}
\newtheorem{defi}[teo]{Definition}
\newtheorem{example}[teo]{Example}

\theoremstyle{remark}
\newtheorem{rem}{Remark}

\usepackage{xcolor}
\usepackage{url}

\usepackage{tikz-cd}
\usepackage{pgf}

\newcommand{\B}{\mathrm{B}}

\newcommand{\Ba}{\mathcal{B}}

\newcommand{\bbar}{\overline}

\newcommand{\CC}{\mathcal{C}}

\newcommand{\Cart}{\mathrm{Cart}}

\newcommand{\cat}{\mathop{\mathrm{cat}}}
\newcommand{\ccat}{\mathop{\mathrm{ccat}}}
\newcommand{\cD}{\mathrm{cD}}

\newcommand{\cTC}{\mathop{\mathrm{cTC}}}

\newcommand{\D}{\mathcal{D}}

\newcommand{\Dc}{\mathcal{D}}

\newcommand{\Di}{\mathrm{cD}}

\newcommand{\E}{\mathcal{E}}

\newcommand{\F}{\mathcal{F}}

\newcommand{\I}{\mathcal{I}}

\newcommand{\id}{\mathrm{id}}

\newcommand{\op}{\mathrm{op}}

\newcommand{\opCart}{\mathrm{opCart}}

\newcommand{\Sets}{\textbf{\textsf{Sets }}}

\newcommand{\Top}{\textbf{\textsf{Top }}}

\newcommand{\Uc}{\mathcal{U}}

\newcommand{\wtilde}{\widetilde}

\newcommand{\Z}{\mathbb{Z}}

\begin{document}


\title[Homotopy invariants]{Homotopy invariants  in small categories}
\author[I. Carcac\'ia-Campos, E.~Mac\'ias-Virg\'os, D.~Mosquera-Lois]{
	%
	I. Carcac\'ia-Campos,
	%
	E.~Mac\'ias-Virg\'os
	\and
	D.~Mosquera-Lois
	}

	\thanks{
	}
	
	\address[A1,A2,A3]{CITMAga
	Universidade de Santiago de Compostela, 15782-Spain}
	\email[A1,A2,A3]
	{\tt isaac.carcacia@rai.usc.es, quique.macias@usc.es, david.mosquera.lois@usc.es}

\maketitle

\begin{abstract}
Tanaka   introduced a notion of  {\em Lusternik} {\em Schnirelmann category}, denoted $\ccat\CC$, of a small category $\CC$. Among other properties, he proved an analog of Varadarajan's theorem for fibrations, relating the LS-categories of the total space, the base and the fiber.

In this paper we  recall the notion of {\em homotopic distance} $\Di(F,G)$ between two functors $F,G\colon \CC \to \D$, later introduced by us, which has $\ccat \CC=\Di(\id_\CC,\bullet)$ as a particular case. We consider another particular case, the distance $\Di(p_1,p_2)$ between the two projections  $p_1,p_2\colon \CC\times \CC \to \CC$, which we call the {\em categorical complexity} of the small category $\CC$. Moreover, we define the {\em higher categorical complexity} of a small category and we show that it can be characterized as a higher distance. 

We prove the main properties of those  invariants. As a final result we prove a Varadarajan's theorem for the homotopic distance  for Grothendieck bi-fibrations between small categories.

 \end{abstract}
\setcounter{tocdepth}{1}


\section*{Introduction}
In the thirties, Lusternik and Schnirelmann introduced a homotopic invariant of manifolds which bounded from the below the number of critical points of any smooth function defined on them (see \cite{CLOT}). For a topological space $X$, this homotopic invariant is known as the Lusternik-Schnirelmann category of $X$, denoted $\cat(X)$. 
\comment{The computation of the Lusternik-Schnirelmann category remains as an open problem for a wide amount of families of spaces (see \cite{CLOT}).}  

Twenty years ago, M. Farber introduced the topological complexity of a topological space (\cite{FARBER}). It is a homotopic invariant which measures the difficulty in finding a motion planning algorithm on the space under consideration. 

Both, the Lusternik-Schnirelmann category and the topological complexity are instances of a more general homotopic invariant called {\em homotopic distance between maps}, introduced by us in  \cite{MAC-MOSQ}. It provides a way to relate and extend several other  homotopic invariants (see \cite{MEDIT}).

Recently, it has grown the interest in developing topological and homotopic invariants in the setting of small categories. For example, this is the case of the Euler characteristic by Leinster (\cite{Leinster_1,Leinster_2}) or the integration with respect to the Euler-Poincar\'e characteristic by Tanaka (\cite{Tanaka_Euler}).  The definition of homotopy between functors by M-J. Lee (\cite{LEE}) made it possible for Tanaka (\cite{TANAKA}) to extend the Lusternik-Schnirelmann category to the setting of small categories. Later, two of the authors extended the notion of homotopic distance to the context of categories (\cite{MAC-MOSQ-FUNCT}).

In this work, we continue the study begun in \cite{MAC-MOSQ-FUNCT}. First, we introduce a novel homotopic invariant for small categories: a notion of {\em higher homotopic distance}. Second, we study fibrations between small categories culminating with a result in the spirit of Varadarajan's theorem (\cite{VAR}) for the homotopic distance for Grothendieck bi-fibrations between small categories (Theorem \ref{thm:fundamental}), 
which can be seen both as a generalization of \cite[Theorem 4.5]{TANAKA} and as an extension of our previous work (\cite{MAC-MOSQ}) to the context of categories.

\comment{These results may be understood as the categorical counterparts of the some of the ones for the homotopic distance for continuous maps between topological spaces discussed in \cite{MAC-MOSQ}}

The paper is organized as follows. Section \ref{sec:cat_dis} is devoted to fixing notation while recalling some preliminaries about homotopies and homotopic distance in the setting of small categories. Moreover, the novel notion of higher categorical distance is introduced  (Definition \ref{def:higher_cat_distanc}) and it is proved to be a particular case of a homotopic distance (Theorem \ref{thm:higher_cat_distance}). In Section \ref{sec:bifibrations} bi-fibrations are presented in this setting from scratch. The exposition is intended to be self-contained so no previous knowledge from the reader is necessary on the topic. Section \ref{sec_homotop_fiber} addresses the equivalence of fibers whose base objects are connected by an arrow (Theorem \ref{EQUIVALE}).  In Section \ref{sec:varadarajan_thm}, we state and prove the main result (Theorem \ref{thm:fundamental}).

\section{Categorical distance between functors} \label{sec:cat_dis}

We will assume that all categories are small unless stated otherwise. If $\E$ is a category, we also denote by $\E$ its set of objects, and by $\E(e_1,e_2)$ the set of arrows between the objects $e_1,e_2\in \E$. If $P\colon \E \to \Ba$ is a functor, we denote by $Pe$ and $P\phi$ the image of the object $e$ and the arrow $\phi$, respectively.

\subsection{Homotopies between functors}

We recall the notion of homotopy between functors introduced by  Lee (\cite{LEE,LEE2}). 

\begin{defi}\label{def:interval_category}
	The {\em interval category $\mathcal{I}_m$} of length $m\geq 0$ consists of $m+1$ objects with zigzag arrrows,
	 $$0 \longrightarrow 1 \longleftarrow 2 \longrightarrow \cdots \longrightarrow (\longleftarrow) m.$$
\end{defi}

Given two small categories $\CC$ and $\Dc$ we denote its product by $\CC \times \D$. Recall that the objects of  $\CC \times \D$ are pairs of objects in $\CC$ and objects in $\Dc$, and its arrows are products of arrows in  $\CC$ and arrows in $\Dc$.

\begin{defi}
	Let $F,G\colon \CC\to \Dc$ be two functors between small categories. We say that $F$  and $G$ are {\em homotopic}, denoted by $F \simeq G$,  if, for some $m\geq 0$, there exists a functor $H\colon \CC \times \I_m \rightarrow \D$, called a homotopy (of length $m$), such that $H_0=F$ and $H_m=G$. 
\end{defi}

Alternatively, the notion of homotopy between functors can be defined as follows. 
\begin{prop}
	The functors $F,G\colon \CC\to \Dc$  are  homotopic  if and only if there is a finite sequence of functors $F_0,\dots,F_m\colon \CC\to \D$, with $F_0=F$  and $F_m=G$, such that for each $i\in \{0,\ldots,m-1\}$ there is a natural transformation either between $F_i$ and $F_{i+1}$ or between $F_{i+1}$ and $F_i$. 
\end{prop}

That both definitions are equivalent follows from the following Lemma.  

\begin{lemma}There is a natural transformation $\Phi\colon F\Rightarrow G$ if and only if there exists a homotopy $H\colon \CC\times \I_1 \to \D$ such that $H_0=F$ and $H_1=G$.
\end{lemma}

\begin{proof}
 For an object, we define
$H(c,0)=Fc$, $H(c,1)=Gc$. For an arrow $f\colon c \to c'$ we define
$H(f\times\id_0)=Ff$, $H(f\times\id_1)=Gf$ and, for the only arrow $s\colon 0 \to 1$ in $\I_1$, we define
$$H(c\times s)=\Phi_c\colon Fc \to Gc.$$
\end{proof}

The homotopy relation defined above is an equivalence relation. Also,  it behaves well with respect to compositions, that is, if $F\simeq F'$ and $G \simeq G'$, then  $F \circ G \simeq F' \circ G'$ whenever $F \circ F'$ and $G \circ G'$ make sense. 

\subsection{Categorical distance between functors}

We introduce now the categorical homotopic distance between functors (see  \cite{MAC-MOSQ-FUNCT}). 


\begin{defi}
Let $\CC$ be a small category, a family $\{{\Uc}_{i}\}_{i \in I}$ of subcategories of $\CC$ if a geometric cover of $\CC$ if for every chain of arrows
$$c_0 \xrightarrow{\alpha_1}  \cdots  \xrightarrow{\alpha_m}  c_m$$ in $\CC$ there is some $i \in I$ such that the chain lies in ${\Uc}_i$.
\end{defi}

There are at least three reasons for adopting this definition: first, just covering objects would leave the arrows between different subcategories uncovered; second, an arbitrary covering of the arrows would leave the compositions of arrows between different subcategories uncovered; third, geometric coverings correspond to coverings of the classifying space of the category, thus making easy to compare the categorical distance with the topological distance of the classifying space.



\begin{defi}
Let $F,G: \CC \rightarrow \D$ be two functors between small categories. The homotopical distance $\Di(F,G)$ beween $F$ and $G$ is the least positive integer $n \geq 0$ such that there is geometric cover $\{{\Uc}_0,...,{\Uc}_n\}$ of $\CC$ such that $F\vert_{{\Uc}_i} \simeq G\vert_{{\Uc}_i}$ for every $ 0 \leq i \leq  n$. 

If there is no such cover we define $\Di(F,G)=\infty$.
 \end{defi}

We  call $\{{\Uc}_0,...,{\Uc}_n\}$   a geometric cover by {\em homotopy domains} for $F$ and $G$.

\subsection{Properties}
The following properties are obvious:
\begin{enumerate}
\item $\Di(F,G)=\Di(G,F)$.
\item $\Di(G,F)=0$ if and only if $F \simeq G$.
\item If $F  \simeq \widehat F$ and $ G \simeq \widehat G$ then $\Di(F,G)=\Di(\widehat F,\widehat G)$.
\end{enumerate}


We finish the section by introducing a notion of higher homotopic distance in the setting of small categories.

\begin{defi}
Let $\CC$ and $\D$ be two small categories and let $\{F_i\}_{i=1} ^n: \CC \rightarrow \D$ be a finite set of functors between them. The {\em higher homotopical distance} $\cD(F_1,...,F_n)$ is the least  integer $m \geq 0$ such that there is a geometric cover $\{{\Uc}_0,...,{\Uc}_m\}$  that satisfies  $F_i\vert_{{\Uc}_k} \simeq F_j\vert_{{\Uc}_k}$ for every $i,j \in \{1,...,n\}$ and $0\leq k \leq m$. 

If there is no such cover we define  $\cD(F_1,...,F_n)=\infty$.
\end{defi}


\subsection{LS-Category} The following notion of ``categorical LS-category'' is due to Tanaka \cite{TANAKA}.

A subcategory $\Uc$ of a small category $\CC$ is called {\em $0$-categorical} if the inclusion functor $\iota\colon\Uc\rightarrow \CC$ is homotopic to a constant functor.
The geometric cover $\{{\Uc}_0,...,{\Uc}_m\}$ of the small category $\CC$ is called {\em categorical} if every ${\Uc}_i$ is $0$-categorical.

\begin{defi}
Let $\CC$ be a small category. We define the (normalized) Lusternik-Schnirelmann category of $\CC$, denoted by   $\ccat(\CC)$, as  the least integer $n \geq 0$ such that there is a categorical cover of $\CC$. 

If there is no such cover we define $\ccat(\CC)$ as $\infty$.
\end{defi}

We state a result from \cite{MAC-MOSQ} and provide a proof since we will make use of it in the proof of Theorem \ref{thm:fundamental}.

\begin{prop}\label{BASE}
For every connected small category $\CC$ we have that $$\ccat(\CC)=\cD(\id_{\CC},\bullet)$$ where $\bullet$ is any constant functor.
\end{prop}

\begin{proof}
Let $\{{\Uc}_{0},\dots,{\Uc}_n\}$ be a geometric cover satisfying that $n=\ccat(\CC)$ and $\iota_{{\Uc}_i} \simeq \bullet_i$ for some constant functors $\bullet_i$ and every $i \in \{0,...,n\}$. We have to prove that the constant functors $\bullet_i$ and $\bullet$ are homotopic. But this follows because there is some arrow which connectes the objects $\bullet_i$ and $\bullet$, and thsi arrow defines a natural transformation between the corresponding constant functors.

It is now evident that $\id\vert_{U_i}=\iota_{{\Uc}_i} \simeq \bullet$ so $\{{\Uc}_0,\dots,{\Uc}_n\}$ is a homotopy domain for $\id_{\CC}$ and $\bullet$. So we have proved that $\ccat(\CC)\geq \cD(\id_{\CC},\bullet)$.

Alternatively if ${\Uc}_0,\dots,{\Uc}_n$ are homotopy domains for $\id_{\CC}$ and $\bullet$ it follows that $\id|_{{\Uc}_i}=\iota_{{\Uc}_i} \simeq \bullet$, so $\cD(\id_{\CC},\bullet) \geq \ccat(\CC)$.
\end{proof}

\begin{cor}
The small category  $\CC$ is contractible if and only if 
$\ccat(\CC)=0$.
\end{cor}
\begin{example}If the small category $\CC$ admits finite products, then $\CC$ is contractible.
\end{example}
\begin{proof}
Fix an object $c_0$ in $\CC$ and define the functor
$c_0\times -\colon \CC \to \CC$ which sends the object $c$ into $c_0\times c$, and the morphism $f\colon c_1\to c_2$ into $\id_{c_0}\times f$. Denote by $c_0\colon \CC \to \CC$ the constant functor.

Then there are natural transformations (that is, homotopies)
$$\id_{\CC} \Leftarrow c_0\times - \Rightarrow c_0$$ 
given by the two projections
$p_1\colon c_0\times c \to c_0$
and $p_2 \colon c_0\times c \to c$.
\end{proof}

Let $\CC$ be a category and let  $c_0$ be an object in $\CC$. There are two functors $i_1, i_2\colon \CC \rightarrow \CC \times\CC$ defined as follows:
\begin{itemize}
\item For every object $c$ in $\CC$ we have that $i_1(c)=(c,c_0)$ e $i_2(c)=i_2(c_0,c)$.
\item Let $f$ be a morphism in $\CC$,then  $i_1(f)=(f,\id_{c_0})$ e $i_2(f)=(\id_{c_0},f)$.
\end{itemize}

\begin{prop}[{\cite[Proposition 6]{MAC-MOSQ-FUNCT}}]
Let $\CC$ be a small category and $C$ an objet in $\CC$. We claim that:
$$\ccat(\CC)=\cD(i_1,i_2).$$
\end{prop}

\subsection{Categorical complexity}

If $\CC$ is a small category, we define the diagonal functor $\Delta\colon \CC \rightarrow \CC  \times \CC $ as the  functor that takes every object $c$ to $\Delta(c)=(c,c)$ and that takes every morphism $f$ to $\Delta(f)=(f,f)$.

\begin{defi}
\begin{enumerate}\item
We say that the subcategory $\Uc$ of  $\CC \times \CC$ is a {\em Farber subcategory} if there is a functor $F\colon  \Uc \rightarrow \CC$ such that $\Delta \circ F \simeq \iota_U$.
\item
The (normalized) {\em categorical complexity} of $\CC$, denoted by $\cTC(\CC)$ is the least integer $n \geq 0$ such that there is a geometric cover $\{{\Uc}_0,\dots,{\Uc}_n\}$ of $\CC \times \CC$  by Farber subcategories. 
\end{enumerate}
If there is no such  cover we define $\cTC(\CC)$ as $\infty$.
\end{defi}

\begin{prop}[{\cite[Theorem 1]{MAC-MOSQ-FUNCT}}]
Let $\CC$ be a small category category, we have that $$\cTC(\CC)=\cD(p_1,p_2),$$ 
where $p_1,p_2\colon \CC \times \CC \to \CC$ are the projections.
\end{prop}

\subsection{Higher categorical complexity} We sketch now a definition of a higher topological complexity, analogous to that existing in the topological setting (\cite{RUDYAK}).

Let $\CC$ be a small category. We denote by $\CC^n$ the product   $\CC   \times \stackrel{n}{ \cdots}\times \CC$. The higher $n$-diagonal functor $\Delta_n\colon \CC\rightarrow \CC ^n$ is  the   functor that takes every object $X$ in $\CC$ to $\Delta_n X=(X,...,X)$ and that takes every morphism $f$ to $\Delta_n f=(f,...,f)$.

\begin{defi}\label{def:higher_cat_distanc}
\begin{enumerate}
\item
We say that the subcategory $\iota\colon\Omega \hookrightarrow \CC^n$ is  a {\em higher $n$-Farber subcategory} if there is a functor $F\colon \Omega \rightarrow \CC$ wich is a right homotopy inverse of $\Delta_n$, that is,   $\Delta_n \circ F \simeq i_\Omega$.
\item
The {\em  (normalized) higher $n$-categorical complexity} of $\CC$, denoted by
 $\cTC_n(\CC)$ is the least integer number $n \geq 0$ such that there is a geometric cover $\{\Omega_0,...,\Omega_n\}$ of $\CC^n$  by higher Farber subcategories. 
\end{enumerate}
\end{defi}

The following result guarantees that higher categorical complexity can be seen as a homotopic distance.

\begin{teo}\label{thm:higher_cat_distance}
The higher topological complexity equals the higher categorical distance between the projections $p_i\colon \CC^n \to \CC$, $1\leq i\leq n$, that is,
$$\cTC\nolimits_n(\CC)=\Di(p_1,...,p_n).$$
\end{teo}

\begin{proof}
( $\leq$)
Let $\Di(p_1,\dots,p_n)=m$. Let $\{U_0,...,U_m\}$ be a geometric cover of $\CC^n$ by homotopy domains, that is, subcategories $U_k$ such that $p_i\vert_{U_k} \simeq p_j\vert_{U_k}$ for all $i,j\in\{0,\dots,n\}$. We will see that they are also higher Farber subcategories. 

Indeed, for every  $U_k$ we have that the first projection $p_1\colon U_k\to \CC$ verifies $\Delta_n\circ p_1 = \iota_{U_k}$. To check that, for every $i \in \{1,...,n\}$ consider the homotopy $H_i\colon U_k\times \I_{m_i} \rightarrow \CC$  between $p_1\vert_{U_i}$ and $p_i\vert_{U_i}$. We can normalize all the homotopies by taking $m=\max \{m_i\}_{i=0}^n$ and extending $H_i\colon U_i \times \I_m \rightarrow \CC$ by identities if $j\geq m_i$. Now we can define a new homotopy 
$$G: U_k \times \I_m \rightarrow \CC$$ as 
$$G(c_1,c_2,\dots,c_n,j)=(H_1(c_1,\dots,c_n),H_2(c_1,...,c_n,j),...,H_n(c_1,...,c_n,j).$$ 

Let us check that $G$ is a homotopy between $\Delta \circ p_1$ and the inclusion $\iota_{U_k}$: 
\begin{align*}
    &\quad G(c_1,\dots,c_n,0)\\&=(H_1(c_1,\dots,c_n,0),\dots,H_n(c_1,\dots,c_n,0))\\
    &=(p_1(c_1,\dots,c_n),\dots,p_1(c_1,\dots,c_n))\\&=(c_1,\dots,c_1)\\
    &=\Delta_np_1(c_1,\dots,c_n),
\end{align*}
and 
\begin{align*}
    &\quad G(c_1,c_2,\dots,c_n,m)\\&=G(H_1(c_1,\dots,c_n,m),\dots,H_n(c_1\dots,c_n,m))\\&=(p_1(c_1,c_2,\dots, c_n),\dots,p_n(c_1,...,c_n))\\&=(c_1,c_2,\dots,c_n).
\end{align*}
This shows that $\cTC\nolimits_n(\CC)\leq m$.

($\geq$)
Now, let $\cTC\nolimits_n(\CC)=m$. Let $\{\Omega_0,...,\Omega_m\}$ be a geometric cover of $\CC^n$ by Farber subcategories. Let us fix some index $k\in\{0,...,n\}$ and let us see that $p_i\vert_{U_k} \simeq p_j\vert_{U_k}$ for every $i,j \in \{1,...,n\}$. 

Since $U_k$ is a Farber subcategory, there is a functor $F\colon U_k \rightarrow \CC$ such that $\Delta \circ F  \simeq \iota_{U_k}$. Hence,  there is a homotopy $H\colon U_k \times \I_m \rightarrow \CC^n $ such that $H_0=\Delta_n \circ F $ and $H_m=\iota_{U_k}$. We take the functor $K\colon U_k \times \I_{2m} \rightarrow \CC$
given by
$$K(c_1,\dotsc_n,l)= \begin{cases}
             p_i \circ H(c_1,...,c_n,m-l) &  \text{ if  }\quad  0 \leq l \leq m \\
             p_j \circ H(c_1,...,c_n,l-m) &  \text{if}  \quad  m \leq l \leq 2m \\
             \end{cases}
$$
The functor $K$ is well-defined because, for  $l=m$:
\begin{align*}
&\quad p_i \circ H (X_1,...,X_n,0)\\
= & p_i \circ \Delta \circ F (c_1,...,c_n)\\
=&p_i(F(c_1,\dots,c_n),\dots,F(c_1,\dots, X_n))\\
=&F(c_1,\dots,c_n)\\
=&p_j \circ \Delta \circ F (c_1,...,c_n)\\
=&p_j \circ H (X_1,...,X_n,0).
\end{align*} 

Moreover $K$ is a homotopy between $p_i$ and $p_j$:
\begin{align*}
K(c_1,\dots,c_n,0)= &p_i \circ H (c_1,\dots,c_n,m) \\
= &p_i \circ \iota_{U_k}(c1,\dots,c_n)=p_i(c_1,\dots,c_n),
\end{align*}
and
\begin{align*}
K(c_1,\dots,c_n,2m)= &p_j \circ H (c_1,\dots,c_n,m) \\
= &p_j \circ i_{U_k}(c1,\dots,c_n)=p_j(c_1,\dots,c_n).
\end{align*}
We have proven that $\Di(p_1,...,p_n)\leq m$.
\end{proof}

\subsection{Properties}
We recall several basic properties of the categorical distance. They serve to give short proofs of many results.

\begin{prop}[{\cite[Propositions 7 and 8]{MAC-MOSQ-FUNCT}}]
\begin{enumerate}
    \item  
Let $\CC$, $D$ and $\Ba$ be three small categories and let $F,G\colon \CC \rightarrow \D$ and $H:\mathfrak{D} \rightarrow \Ba$ be three functors. We have that:
$$\cD(H \circ F,H \circ G) \leq \cD(F,G).$$ 
\item
Analogously, let $\CC$, $\D$ and $\Ba$ be three small categories and let $F,G\colon \CC \rightarrow \D$ and $H:\Ba \rightarrow \CC$ be three functors. We have that:
$$\cD(H \circ F,H \circ G) \leq \cD(F,G).$$ 
\end{enumerate}
\end{prop}

\begin{cor} 
Let $\CC$ be a connected small category. Then we have that $$\ccat(\CC) \leq \cTC(\CC).$$
\end{cor}
\begin{proof}$$\ccat(\CC)=\cD(\id_{\CC},C_0)=\cD(p_1\circ i_1,p_2 \circ i_1) \leq \cD(p_1,p_2)=\cTC(\CC).$$
\end{proof}

In what follows we note that category and categorical complexity are in some sense dual and extreme cases of categorical distance.

\begin{cor} 
Let $F,G\colon\CC \rightarrow \D$ be two functors between small categories, Then we have that
$$\cD(F,G) \leq \cTC(\D).$$
\end{cor}
\begin{proof}
We  define the functor $(F,G)\colon\CC \rightarrow \D \times \D$ such that  $p_1  \circ (F,G)=F$ and $p_2\circ(F,G)=G$. Therefore we have that
$$\cD(F,G)=cD(p_1 \circ(F,G), p_2 \circ (F,G)) \leq \cD(p_1,p_2)=\cTC(\D).$$ 
\end{proof}

\begin{prop} [{\cite[Theorem 2]{MAC-MOSQ-FUNCT}}]
Let $F,G\colon\CC \rightarrow \D$ be two functors between small categories, Then 
$$\cD(F,G) \leq \ccat(\CC).$$
\end{prop}

\section{Bi-fibrations} \label{sec:bifibrations}
We recall the notions of cartesian morphism and Grothendieck fibration. We follow the references  \cite{DAZIN}, \cite[Appendix A]{KOEN}, \cite{STREITCHER}, \cite{VISTOLI} and \cite[Chapter 12]{PENNER}, as well as \cite[Section 4]{TANAKA}.

\subsection{Cartesian arrows}

Let $P\colon \E \to \Ba$ be a functor.

 \begin{defi}
The morphism $\phi\in \E(e_1,e_2)$ is {\em cartesian} (with respect to $P$) if, for every arrow $\beta \in\E(e,e_2)$ and  every arrow
 $\bbar\alpha\in\Ba(Pe,Pe_1)$  such that $P\phi\circ \bbar\alpha=P\beta$,
there exists a unique arrow $\alpha\in\E(e,e_1)$ such that $\phi\circ\alpha=\beta$ and \Q{$P\alpha=\bbar\alpha$} \DM{(see Diagram (\ref{diag:def_cartesian}))}.
\begin{equation}\label{diag:def_cartesian}
\begin{tikzcd}
& e_1 \ar{d}{\phi} \\
 e \arrow[ur,dashrightarrow,"\alpha"]\arrow[r,"\beta"] & e_2 
\end{tikzcd}
\quad \quad
 \begin{tikzcd}
& Pe_1 \arrow[d,"P\phi"] \\
Pe\arrow[ur,"{\bbar\alpha}"]\arrow[r,"P\beta"] & Pe_2
\end{tikzcd}
\end{equation}
\end{defi}


 The name {\em cartesian} seems to have its origin in the following characterization, whose proof is left to the reader:
 \begin{prop}The morphism $\phi$ is cartesian if and only if the following commutative square in the category of \Sets is a pullback:
$$\begin{tikzcd}
\E(e,e_1)\arrow[d,"{P(-)}"]\arrow[r,"{\phi\circ-}"]&\E(e,e_2)\arrow[d,"{P(-)}"]\\
\Ba(Pe,Pe_1)\ \ \arrow[r,"{P(\phi)\circ-}"]&\ \ \Ba(Pe,Pe_2)
\end{tikzcd}
$$
\end{prop}

\begin{example}Let $P\colon \CC \to \CC$ the identity functor. Then any arrow in $\CC$ is cartesian.
This is obvious because $P\alpha=\alpha$.
\end{example}

\begin{example}
Let $P\colon \CC \to \bullet$ be the constant functor. An arrow $\phi$ in $\CC$ is $P$-cartesian if and only if $\phi$ is an isomorphism.

Let us check it. From the diagram 
$$
\begin{tikzcd}
& e_1 \ar{d}{\phi} \\
 e_2 \arrow[ur,dashrightarrow,"\alpha"]\arrow[r,"\id_{e_2}"'] & e_2 
\end{tikzcd}
\quad \quad
 \begin{tikzcd}
& \bullet \arrow[d,"\id_\bullet"] \\
\bullet\arrow[ur,"{\id_\bullet}"]\arrow[r,"\id_\bullet"'] & \bullet
\end{tikzcd}
$$
we obtain an arrow such that $\phi\alpha=\id_{e_2}$. Now, $\phi(\alpha\phi)=\phi$ means that we have the diagram
$$
\begin{tikzcd}
& e_1 \ar{d}{\phi} \\
 e_1 \arrow[ur,dashrightarrow,"\alpha\phi"]\arrow[r,"\phi"'] & e_2 
\end{tikzcd}
\quad \quad
 \begin{tikzcd}
& \ast \arrow[d,"\id_\ast"] \\
\ast\arrow[ur,"{\id_\ast}"]\arrow[r,"\id_\ast"'] & \ast
\end{tikzcd}
$$
and by unicity it must be $\alpha\phi=\id_{e_1}$. So the morphism $\phi$ is an isomorphism.

The converse is immediate.
\end{example}



\begin{example} Posets can be viewed as small categories,  with an arrow existing between two objects if and only if $c_1\leq c_2$. Functors between posets are the order preserving maps $P\colon \E \to \Ba$. For such a functor, a map $e_1\leq e_2$ is cartesian if and only if it verifies the following condition:
if $e\leq e_2$ and $Pe\leq e_1$ then $e\leq e_1$.

This condition ressembles the definition of a {\em down beat point} \cite{BARMAK}. In fact. down beat points are cartesian for any functor.\end{example}


For a study of fibrations in the setting of posets we refer the reader to \cite{Cianci_Ottina}.

\comment{ejemplo de clasificante contractil pero la categoria no, por tanto BC no es quivalente}
\subsection{Fibrations}

\begin{defi}The functor $P\colon \E \to \Ba$ is a {\em fibration} if for any arrow $\bbar\phi\in \Ba(b_1,b_2)$ and any object $e_2\in\E$ such that $Pe_2=b_2$, there exists a cartesian arrow $\phi\in \E(e_1,e_2)$ such that $P\phi=\bbar\phi$.
\end{defi}

The arrow $\phi$ is called a {\em cartesian lift} of $\bbar\phi$ with codomain $e_2$.

\begin{defi}
If $P\colon \E \to \Ba$ is a fibration, we define the {\em fiber} of $P$ over $b\in\Ba$ as the subcategory of $\E$ with objects $e\in E$  such that $Pe=b$, and with arrows 
$\nu\in \E(e_1,e_2)$ such that $P\nu=\id_b$. These arrows are called {\em vertical} arrows.
\end{defi}

The cartesian lift of a given $\bbar\phi\colon b_1 \to b_2$ with a given codomain $e_2$ is unique, up to a unique vertical arrow. This follows from the unicity in the definition of cartesian arrow.

By using the axiom of choice, we can take a particular lift, which will be denoted by
$$\Cart(\bbar\phi, e_2)\colon \bar\phi^*e_2 \to e_2.$$
This particular choice defines a functor ${\bar\phi}^*\colon \E_{b_2} \to \E_{b_1}$, where the image  of a vertical arrow $\nu\in \E_{b_2}$ is given by the unique arrow ${\bbar\phi}^*\nu$ making the following diagram commute:

$$\begin{tikzcd}
\bbar\phi^*e_2\arrow[d,"{\bbar\phi^*\nu}"',dashed]\arrow[r,"{\Cart(\bbar\phi, e_2)}"]&e_2\arrow[d,"\nu"]\\
\bbar\phi^*e_2'\ \ \arrow[r,"{\Cart(\bbar\phi, e_2')}"]&\ \ e_2'.
\end{tikzcd}
$$
It corresponds to the diagram
$$
\begin{tikzcd}
&& {\bbar\phi}^*e_2' \ar{d}{\Cart} \\
{\bbar\phi}^*e_2\arrow[rru,dashrightarrow] \arrow[r] &e_2\arrow[r,"\nu"']& e_2 
\end{tikzcd}
\quad \quad
 \begin{tikzcd}
&& b_1 \arrow[d,"\phi"] \\
b_1\arrow[urr,"{\id}"]\arrow[r,"\phi"'] &b_2\arrow[r,"\id"]& b_2
\end{tikzcd}
$$

The functoriality of ${\bbar\phi}^*$ follows again from unicity. It is called the {\em pullback functor}.



\begin{example}
Every group $G$ can be considered as a category with one object,
where the  arrows are the elements $g\in G$, and the composition is given by the
operation in $G$. A group homomorphism $F\colon H \to G$ can be considered as a functor.
An arrow in $H$ is always cartesian; hence $F$ is a fibration if and only if is surjective. The fiber is the kernel.
\end{example}

\begin{example}\label{FUNDAM} Let $\CC^{\I_1}$ be the category whose objects are the arrows in $\CC$, and whose arrows are the commutative squares. The {\em codomain} functor $P\colon \CC^{\I_1}\to \CC$ is a fibration if and only if $\CC$ has enough pullbacks. The cartesian arrows are the pullbacks in $\CC$. It is called the {\em fundamental fibration}.
\end{example}

We introduce another example of a fibration in the setting of locally small categories.

\begin{example} Let   $U\colon \Top \to \Sets$ be the forgetful functor from topological spaces to sets. It associates to each topological space  $X$ its underlying set $UX$, and to each continuous map the set map itself. It  is a fibration. The $U$-cartesian maps are the continuous maps $f\colon e_1 \to e_2$ such that the topology on $e_1$ is the {\em initial} topology, that is, the smallest (coarsest) topology making $f$ continuous.
\end{example}

\comment{\begin{example}A morphism in the category $\CC$ is cartesian for the constant functor $\CC \to \ast$ if and only if is an isomorphism.
\end{example}}

\subsection{Op-cartesian arrows and op-fibrations}
Let $P^\op \colon \E^\op \to \Ba^\op$ be the opposite functor of $P\colon \E \to \Ba$.
\begin{defi}The arrow $\varphi\colon e_1 \to e_2$ in $\E$ is op-cartesian for the functor $P$ if the opposite arrow $\varphi^\op$ in $\E^\op$ is cartesian for $P^\op$. Explicitly, that means that for any given $\beta\in \E(e_1,e)$ and any given $\bbar\alpha\in \Ba(Pe_2,Pe)$ such that $\bbar\alpha\circ P\varphi=P\beta$, there exists a unique $\alpha\in\E(e_2,e)$ such that $\alpha\circ\varphi=\beta$ and $P\alpha=\bbar{\alpha}$, as in the following diagram:
$$\begin{tikzcd}
e_1 \ar[d,"\varphi"']\arrow[r,"\beta"] &e \\
 e_2 \arrow[ur,dashrightarrow,"\alpha"']&  
\end{tikzcd}
\quad \quad
 \begin{tikzcd}
 Pe_1 \arrow[d,"P\varphi"'] \arrow[r,"P\beta"] &Pe \\
Pe_2\arrow[ur,"{\bbar\alpha}"']& 
\end{tikzcd}
$$
\end{defi}

\begin{defi}A functor $P\colon \ E \to \Ba$ is an {\em op-fibration} if for any map $\bbar\varphi\colon b_1\to b_2$ in $\Ba$, and for any object $e_1$ in $\E$ with $Pe_1=b_1$, there exists an op-cartesian arrow $\varphi\colon b_1 \to b_2$ such that $P\varphi=\bbar\varphi$.
\end{defi}

Again, this {\em op-cartesian lifting} is unique up to a unique vertical isomorphism. Then we can choose some particular lifting
$$\opCart(\bbar \varphi, e_1)\colon e_1 \to \bbar\varphi_*e_1$$
so defining a functor
$$\bbar\varphi_*\colon \E_{b_1} \to \E_{b_2}$$
bewteen the fibers.


\begin{example}The codomain functor $P\colon \CC^{\I_1} \to \CC$ of Example \ref{FUNDAM} is always an op-fibration.
\end{example}

\begin{example} Analogously, the {\em domain} functor $P\colon \CC^{\I_1} \to \CC$ is an op-fibration if and only if $\CC$ has enough push-outs.
\end{example}

\begin{defi}We say that the functor $P\colon \E \to \Ba$ is a {\em bi-fibration} if it is both a fibration and an op-fibration.
\end{defi}

We prefer this terminology instead of ``fibration'' and ``cofibration'', see \cite{nLab} for a discussion.

\begin{example} The forgetful functor  $U\colon \Top \to \Sets$ from topological spaces to sets is an op-fibration. Op-cartesian maps are the continuous maps $e_1\to e_2$ where $e_2$ has the {\em quotient topology}.
\end{example}

\begin{example}Let $\varphi\colon K \to K'$  be a simplicial map of simplicial complexes.
Let $\mathcal{X}(\varphi) \colon \mathcal{X}(K) \to \mathcal{X}(K')$ be the corresponding map between the associated face posets. A poset can be considered as a small category with the obvioius arrows. Then the functor $\varphi$ is a fibration, but it is not an op-fibration (\cite[Example 13.5]{PENNER}). 
\end{example}

\begin{example}\label{productofib}
Take a base category $\Ba$ and another category $\CC$, then the first projection $P_1:\Ba\times \CC \rightarrow \B$ is a bi-fibration. In fact, given an arrow $\bbar\phi\colon b_1 \to b_2$ in $\CC$, and an object $(b_2,c)$ in $\Ba\times \CC$, the morphism $\phi=\bbar\phi\times \id_c\colon (b_1,c) \to (b_2,c)$ is a cartesian lifting. In order to prove it we only must take into account the diagram:
$$
\begin{tikzcd}
& {(b_1,c)} \arrow[d, "{\phi}"] &  \\
{(b,c')}  \arrow[ur, "{w\times v}", dashed]\arrow[r, "{u\times v}"'] & {(b_2,c)}                                      
\end{tikzcd}
\quad\quad
\begin{tikzcd}
& b_1 \arrow[d, " \bbar\phi"] \\
 b \arrow[ur, "w"] \arrow[r, "u"'] & b_2 
\end{tikzcd}
$$
Analogously, $P_1$ is an op-fibration.
\end{example}

\subsection{Further examples}

\begin{example}
Let $\E$ be the category with objects the integers $n\in\mathbb{Z}$ and a unique morphism  $m \rightarrow n$ if and only if $m \leq n$.
$$
\cdots \to -1 \to 0 \to 1 \to \cdots ;
$$
and let $\Ba$ be the monoid of the natural numbers $n\geq 0$ as a category.
$$
\begin{tikzcd}
\bullet \arrow["n \in \mathbb{N}"', loop, distance=2em, in=125, out=55]
\end{tikzcd}
$$
The functor $P\colon \E \to \Ba$ with $P(n)=\bullet$ and $P(m\rightarrow n)=n-m$ is a fibration. In fact, let us consider $n\colon  \bullet \rightarrow \bullet$ a morphism in $\Ba$ and $m$ any object in $\E$ ; then the morphism $m-n\rightarrow m$ covers $n$. Moreover it is $P$-cartesian as indicated in the following commutative diagrams: 
$$
\begin{tikzcd}
 & m- n\arrow[d] &  \\
l  \arrow[ur, dashed]\arrow[r] & m                                                                 
\end{tikzcd}
\quad
\begin{tikzcd}
&    \bullet \arrow[d, "n"]\\
 \bullet \arrow[ur, "m-l-n"]  \arrow[r, "m-l"'] & \bullet
\end{tikzcd}
$$
Furthermore, $P$ is an op-fibration by a similar argument.
\end{example}

\begin{example}
Let $\E$ be the category with objects the integer numbers $n\in \Z$,  and ziz-zag arrows:
$$\cdots \longleftarrow -2 \longrightarrow -1\longleftarrow 0 \longrightarrow 1 \longleftarrow 2 \longrightarrow  \cdots$$


and let $\Ba$ be the following category 
$$
\begin{tikzcd}
\bbar 0 \arrow[d] \arrow[rd] & \bbar 2 \arrow[ld] \arrow[d] \\
\bbar 3                    & \bbar 1                    
\end{tikzcd}
$$
It is easy to show that the functor $P(n)= \bbar n \mod 4$
is a fibration and an op-fibration.
\end{example}

\section{Lifting properties}
The lifting properties of Grothendieck fibrations were considered by Gray in \cite{GRAY}. 

Recall that we denote by  $\I_n $
the $n$-{\em chain category}   generated by the following diagram:
$$
0 \longrightarrow 1 \longrightarrow \cdots \longrightarrow  n.
$$
In particular, $\I_1$ denotes the category $0 \stackrel{s}{\longrightarrow} 1$ consisting of two objects
and one non-identity morphism $s$. 

Gray proved that the functor $P\colon \E \to \Ba$ is a fibration if and only if homotopies have cartesian liftings. Since we represent a natural transformation ending at $P\circ G$ as a functor $\eta\colon \CC \times \I_1 \to \Ba$ where $\eta_1=P\circ G$, we can state the following, which improves  \cite[Prop.4.2]{TANAKA}:

\begin{prop}
The functor $P\colon \E \to \Ba$ is a fibration if and only for any category $\CC$ and for any functor $H\colon \CC \times \I_1 \to \Ba$ such that $H_1=H\circ i_1= P\circ G$ there exists a functor $\wtilde H \colon \CC\times \I_1 \to\E$ such that $P\circ \wtilde H=H$ and $\wtilde H_1=\wtilde H\circ i_1=G$, as in the following diagram:
$$
\begin{tikzcd}
\CC \arrow[d,"i_1"'] \arrow[r,"G"] & \E  \arrow[d,"P"] \\
\CC\times \I_1 \arrow[ur,"\wtilde H",dashed]\arrow[r,"H"']                   & \Ba                   
\end{tikzcd}
$$
 and moreover   $\tilde H(f\times s)$ is a cartesian arrow for $P$
 for any arrow $f$ in $\CC$.

Here, $i_1\colon \CC \to \CC\times \I_1$ is the functor sending the object $c$ into $(c,1)$, and the arrow $f\colon c_1\to c_2$ into $f\times \id_1$, and   $s$ is the only arrow $0\to 1$ in $\I_1$.
\end{prop}

\begin{proof}
Obviously we must define  $\wtilde H_1=G$.  The problem  is how to define $\wtilde H_0$ in order to have a functor. To do that we have to use the definition of fibered category. As we know, if we have an arrow $\bbar\phi$ in $\Ba$ and we have an object $e_2$ in the fiber of the codomain $b_2$,  we can lift the arrow into a cartesian morphism $\Cart(\bbar\phi,e_2)\colon {\bbar\phi}^*e_2\to e_2$. In particular, we can choose
$\Cart(\id_b,e)=\id_e$. 

Now, if we have the arrow $H(\id_c\times s)\colon H(c,0) \rightarrow H(c,1)$ in $\Ba$ and the object $G(X)\in \E_{H(c,1)}$, there is a cartesian morphism 
$$\Cart H(\id_c\times s))\colon H(\id_c\times s)^*G(c) \rightarrow G(c),$$ so we define $$\wtilde H(c,0)=H(\id_c\times s)^*G(c)$$ and $$\wtilde H(\id_c\times s)=\Cart H(\id_c\times s).$$

Finally, for every morphism $f\colon c_1 \rightarrow c_2$ in $\CC$ we define $\wtilde H(f\times s)$ as the unique cartesian arrow filling the diagram
$$
\begin{tikzcd}
{\wtilde H(c_1,0)} \arrow[d,dashed]\arrow[rr, "{\wtilde H(\id_{c_1}\times s)}"] &  & {\wtilde H(c_1,1)=Gc_1} \arrow[d, "{\wtilde H(f\times \id_1)=Gf}"] \\
{\wtilde H(c_2,0)} \arrow[rr, "{\wtilde H(\id_{c_2}\times s)}"] &  & {\wtilde H(c_2,1)=Gc_2}                            
\end{tikzcd}
\quad
\begin{tikzcd}
{H(c_1,0)} \arrow[d, "{H(f\times \id_0)}"'] \arrow[rr, "{H(\id_{c_1}\times  s)}"] &  & {H(c_1,1)} \arrow[d, "{H(f\times \id_1)}"] \\
{H(c_2,0)} \arrow[rr, "{H(\id_{c_2}\times s)}"]                           &  & {H(c_2,1)}                          
\end{tikzcd}
$$

For the converse statement, let $\bbar\phi\colon b_1 \to b_2$ in $\B$ be an arrow and consider the diagram
$$
\begin{tikzcd}
\bullet \arrow[d,"i_1"'] \arrow[r,"G"] & \E  \arrow[d,"P"] \\
\bullet\times \I_1 \arrow[ur,"\wtilde H",dashed]\arrow[r,"H"']                   & \Ba      \end{tikzcd}
$$
where $g(\bullet)=e_2$ and $H(\id_\bullet\times s)=\bbar\phi$. Then the map $\phi={\tilde H}(\id_\bullet\times s)$ is cartesian and verifies that $P\phi=\bbar\phi$.
This ends the proof.\end{proof}



Analogously, we have the lifting propery for op-fibrations, where the left vertical arrow
is $i_0$ instead of $i_1$:
$$
\begin{tikzcd}
\CC \arrow[d,"i_0"'] \arrow[r,"G"] & \E  \arrow[d,"P"] \\
\CC\times \I_1 \arrow[ur,"\wtilde H",dashed]\arrow[r,"H"']                   & \Ba                   
\end{tikzcd}.
$$

Taking in account this and the previous proposition we have the following general lifting property.

\begin{cor}
Let $P:\E \rightarrow \Ba$ be a bi-fibration. Then the lifting propery holds for any chain category $\I_n$:
$$
\begin{tikzcd}
\CC \arrow[r, "G"] \arrow[d, "i_0"']     & \E \arrow[d, "P"] \\
\CC \times \I_n \arrow[ur,"\wtilde H",dashed]\arrow[r, "H"'] & \Ba              
\end{tikzcd}
$$

\end{cor}

\section{Homotopic fiber} \label{sec_homotop_fiber}
The next result is a crucial one. It proves that in a bi-fibration, two objects $b_1,b_2\in \Ba$ in the base which are connected by an arrow $\bbar\varphi\colon b_1 \to b_2$, have fibers which are homotopically equivalent as small categories.

\begin{teo}[{\cite[Proposition 4.4]{TANAKA}}]\label{EQUIVALE}
Let $P:\E \rightarrow \Ba$ be a bi-fibration.  If $b_1$ and $b_2$ are two objects in $\B$ such that   there is a morphism $u\colon b_1 \to b_2$, then there is a categorical equivalence between the fibers  $\E_{b_1}$ and $\E_{b_2}$.
\begin{proof}
Remember that in a bi-fibration we have both the pushforward functor $u_*\colon \E_{b_1} \to \E_{b_2}$ and the pullback functor $u^*\colon \E_{b_2} \to \E_{b_1} $. 

We defined $u_*$ as follows: for every object $e_1$ in the fiber $\E_{b_1}$ there is a unique object $u_*{e_1}$ in $\E_{b_2}$ such that   $\opCart(u,e_1)\colon e_1 \to u_*e_1$ is the unique op-cartesian lift of $u$. 

Alternatively, for every $e_2$ in $\E_{b_2}$ there is a unique object $u^*e_2\in \E_{b_1}$ such that $\Cart(u,e_2)\colon u^*e_2 \to e_2$   is the unique cartesian lift of $u$.

Now in order to show tha $u_*$ and $u^*$ induce a categorical equivalence it is enough to show that for every object $e_2$ in $\E_{b_2}$ there are natural arrows $\alpha$ and $\beta$ as in the following diagram:
$$
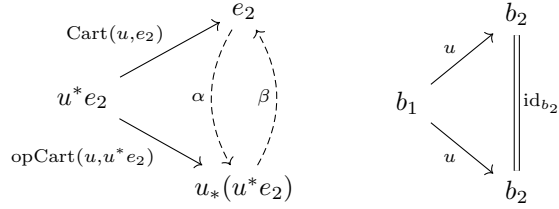
\begin{tikzcd}
                                                 & e_2 \arrow[dd, "\alpha"', dashed, bend right]        \\
u^*e_2 \arrow[ru, "{\Cart(u,e_2)}"] \arrow[rd, "{\opCart(u,u^*e_2)}"'] &                                                   \\
                                                 & u_*(u^*e_2) \arrow[uu, "\beta", dashed, bend right]
\end{tikzcd}
\quad\quad\quad
\begin{tikzcd}
                                  & b_2 \arrow[dd, "\id_{b_2}", equal] \\
b_1 \arrow[ru, "u"] \arrow[rd, "u"'] &                                            \\
                                  & b_2                                         
\end{tikzcd}
$$
and that their compositions are the identities because $\Cart(u,e_2)$ and $\opCart(u,e_1) $ are a cartesian and op-cartesian morphism (respectively). 
Thus, we have a natural isomorphism between the functors $\id_{\E_{b_2}}$ and $u_*\circ u^*$. 

Analogously we have another natural isomorphism between $\id_{\E_{b_1}}$ and $u^*\circ u^*$ using the followings diagrams:
$$
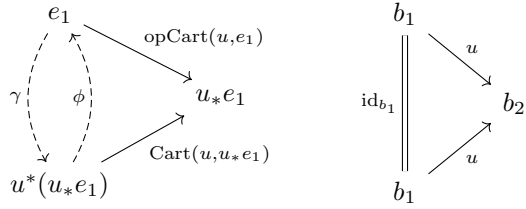
\begin{tikzcd}
e_1 \arrow[rd, "{\opCart(u,e_1)}"] \arrow[dd, "\gamma"', dashed, bend right]  &        \\
 & u_*e_1\\
 u^*(u_*e_1) \arrow[ru,"{\Cart(u,u_*e_1)}"'] \arrow[uu, "\phi", dashed, bend right] &   
\end{tikzcd}
\quad\quad  \quad
\begin{tikzcd} 
b_1 \arrow[rd, "u"] \arrow[dd, "\id_{b_1}"', equal] &   \\                                                           & b_2                                                            \\
b_1 \arrow[ru, "u"']                                          &   
\end{tikzcd}
$$
\end{proof}
\end{teo}

A consequence of the latter theorem is that we can speak about the homotopic invariants of ``the fiber'' of a bi-fibration.

\begin{rem}
In the proof of Theorem \ref{EQUIVALE} we implicitly assumed  that both fibers $\E_{b_1}$ and $\E_{b_2}$ are non empty. This might not be true. However, the Theorem is still true, because for a bi-fibration, if two objects $b_1$, $b_2$  in the base are connected by an arrow, say $\bbar\phi\colon b_1 \to b_2$, then either $\E_{b_1}=\emptyset =\E_{b_2}$ or
$\E_{b_1}\neq \emptyset \neq \E_{b_2}$ simultaneously.

In fact, assume first that $P$ is only a fibration.  If $\E_{b_2}=\emptyset$
the pullback functor is not defined, because we have no codomain $e_2$ to lift the arrow $\bbar\phi$. 
Since a fibration $P\colon \E \to \Ba$ may not be surjective-on-objects (see Example \ref{NOSOBRE}), it may happen that $\E_{b_2}$ is empty,   while $\E_{b_1}$ is not.

However, for a fibration, if $\E_{b_2}\neq \emptyset$ then $\E_{b_1}\neq\emptyset$ too, because to each codomain $e_2\in\E_{b_2}$ we associate the domain ${\bbar\phi}^*e_2$ of $\Cart(\phi,e_2)\in\E_{b_1}$.

Dually, for an op-fibration, if $\E_{b_1}\neq \emptyset$ then $\E_{b_2}\neq \emptyset$, because to each domain $e_1$ we associate the codomain ${\bbar\phi}_*e_1$ of $\opCart(\bbar\phi,e_1)$.

As a consequence, in a bi-fibration there are no arrows connecting objects in the image with objects outside the image, or conversely.
\end{rem}

\begin{example}\label{NOSOBRE} A fibration which is not surjective-on-objects is for instance the constant functor
$P\colon \I_1 \to \I_1$ with $P(0)=P(1)=0$, and $P(s)=\id_0$. We have $\E_{0}=\I_1$ while $\E_{1}=\emptyset$. 
This shows that $P$ is not an op-fibration. In fact, given the arrow $s\colon 0 \to 1$ and the domain $1\in\E_0$  there is no op-cartesian lift of $s$.

Notice that the arrow $s\colon 0 \to 1$ can not be lifted to any cartesian arrow, but this does not contradict the definition of fibration (there are no possible codomains).
\end{example}


\begin{example}
The following functor $P:\E\rightarrow \Ba=\I_1$ is a bi-fibration, but the fibers over $0$ and $1$ are not isomorphic. However, they are homotopically equivalent, as stated in Theorem \ref{EQUIVALE}.

Let $\E$ be the category generated by the following 
diagram:
$$
\begin{tikzcd}
            & \bar{1} \arrow[d, "f"', bend right] \\
0 \arrow[r,"s"'] & 1 \arrow[u, "g"', bend right]      \end{tikzcd}
$$
where $f\circ  g=\id_1$ and $g\circ f=\id_{\bar{1}}$.
Let $\Ba=\I_1$ the category with only one arrow $0\xrightarrow{s}1$.

The functor $P$ is defined as $P(0)=0$, $P(1)=P(\bar{1})=1$, $P(s)=s$ and $P(f)=P(g)=\id_{1}$. It is a bi-fibration because the arrows $\id_0$, $\id_1$, $\id_{\bar 1}$, $s$ and $g\circ s$ are cartesian and op-cartesian.
The fiber over $0$ is the discrete category with one object $0$ and the fiber over $1$ is the following category:
$$
\begin{tikzcd}
\bar{1} \arrow[d, "f"', bend right] \\
1 \arrow[u, "g"', bend right]      
\end{tikzcd}.
$$
It is contractible by the   natural transformation $\alpha$ between the identity and the constant functor $\bar{1}$ given by $\alpha(1)=g$, $\alpha(\bar 1)=\id_{\bar 1}$. In fact, we have
$$
\begin{tikzcd}
1 \arrow[d, "g"] \arrow[rr, "\alpha(1)=g"]          && \bar{1} \arrow[d, "\id_{\bar{1}}"] \\
\bar{1} \arrow[rr, "\alpha(\bar{1})=\id_{\bar{1}}"] && \bar{1}      \end{tikzcd}
\quad
\begin{tikzcd}
 \bar{1} \arrow[d, "f"] \arrow[rr, "\alpha(\bar{1})=\id_{\bar{1}}"]&& \bar{1} \arrow[d, "\id_{\bar{1}}"] \\
 1 \arrow[rr, "\alpha(1)=g"]                                        && \bar{1}  
\end{tikzcd}
$$
\end{example}


\section{Varadarajan's theorem} \label{sec:varadarajan_thm}
In the classical LS-category theory of topological spaces, Varadarajan \cite{VAR} proved a formula relating the categories of the fiber, the base and the total space for a Hurewicz fibration. We generalized this result for the homotopic distance, in \cite[Theorem 6.1]{MAC-MOSQ}.

In \cite[Theorem 4.5]{TANAKA}, Tanaka proved an analogous result for a bi-fibration $P\colon \E \to \Ba$ with fiber $\F$  between small categories,  with path-connected base, namely
\begin{equation}\label{VARSMALL}
\ccat \E +1 \leq (\ccat \Ba +1)\cdot (\ccat \F +1).
\end{equation}

Recall that by Corollary \ref{EQUIVALE}, if the base category $\Ba$ of a bi-fibration $P: \E \rightarrow \Ba$ is path-connected, then every two fibers are homotopy equivalent.
We will refer to any of them as (the homotopy type of)  ``the fiber $\F$ of the bi-fibration''.  


We will both extend our \cite[Theorem 6.1]{MAC-MOSQ} to the context of bi-fibrations of small categories and generalize Tanaka's result, as follows.

\begin{teo}\label{thm:fundamental}
Let $(F,\bbar F)$ and $(G,\bbar G)$ be two morphisms between the bi-fibrations $P\colon \E\rightarrow \Ba$ and $P'\colon \E ' \rightarrow \Ba '$. Let $\Ba$ and $\Ba'$ be path-connected.
Let  $b$ be an object in $\Ba$ such that $F(b)=G(b)=b'$ and let  $F_b,G_b\colon \E_b\rightarrow \E'_{b'}$ be the induced functors between the fibers. Then
$$\Di(F,G)+1  \leq (\Di(F_b,G_b)+1)\cdot (\ccat(\Ba)+1).$$
\end{teo}

\begin{proof}

We assume that $\Di(F_b,G_b)$ and $\ccat\Ba $ are both finite, otherwise the result is trivial.

Let $\ccat(\Ba)=m$, with $\{U_0,...,U_m\}$ a categorical cover of $\Ba$, and let  $\Di(F_b,G_b)=n$, with  $\{V_0,...,V_n\}$  a covering of $\E_b$ by homotopy domains for  $F_b$ and $G_b$. 

For every $i \in \{0,...,m\}$ we have a homotopy $C^i\colon  U_i \times \I_{k_i} \to \Ba$ between the inclusion $U_i\hookrightarrow \Ba$ and a constant functor $\bullet_i$ for some object $\ast_i\in\Ba$. Since $\Ba$ is connected we can assume that all $\bullet_i$ is the same object $\ast$ for all $i$ (see the proof of Proposition \ref{BASE}).

Let $P^{-1}(U_i)$ be the subcategory of $\E$ whose objects are the objects $e\in \E$ with $Pe\in U_i$, and whose arrows are the arrows $\alpha\in \E(e_1,e_2)$ such that $P\alpha$ is an arrow in $U_i$.

By the homotopy lifting property applied to the following diagram:
$$
\begin{tikzcd}
P^{-1}(U) \arrow[d, "i_0"] \arrow[rrrr, hook]                                                                        & {} \arrow[rrr] &                                             &  & \E \arrow[d, "P"] \\
P^{-1}(U)\times \mathcal{I}_{k_i} \arrow[rrrru, "\wtilde C^i", dashed,end anchor={[shift={(1pt,0pt)}]south west} ] \arrow[rr, "{P \times \id}"'] &                & U \times \mathcal{I}_{k_i} \arrow[rr, "C^i"'] &  & \Ba              
\end{tikzcd}
$$
we have a homotopy $\wtilde C^i: U_i \times \mathcal{I}_{k_i}\to \E$ such that ${\wtilde C^i}_0$ is the inclusion $P^{-1}(U_i)\hookrightarrow \E$ and ${\wtilde C^i}_{k_i}$ lies inside the fiber $\E_{b}$.

For each $i\in\{1,\dots,m\}$ and $j\in\{1,\dots,n\}$ we define 
$$
W_{i,j}=P^{-1}(U_i) \cap ({\wtilde C^i}_{k_i})^{-1}(V_j).
$$
We claim that $\{W_{i,j}\}_{0 \leq i \leq m,0\leq j\leq n}$ is a geometric cover of $\E$ such that each $W_{i,j}$ is a homotopy domain for $F$ and $G$.
    
    (1) $\{W_{i,j}\}
    $ 
    is a geometric cover of $\E$.
    
    Let 
    $$
\mathcal{C}\colon  C_1 \longrightarrow  C_2\longrightarrow \cdots \longrightarrow C_l
$$ be a chain in $\E$. Then  we obtain the chain  
$$
P(\CC)\colon P(C_1)  \longrightarrow  P(C_2) \longrightarrow \cdots  \longrightarrow  P(C_l)
$$
    in $\Ba$. Since $\{U_0,...,U_m\}$ is a geometric cover of $\Ba$, there is some $i$ such that the chain $P\CC)$ lies in $U_i$, so  the chain $\CC$ lies in $P^{-1}(U_i)$. Moreover, the functor ${\wtilde C^i}_{k_i}$ is  defined in $\CC$, hence we have a new chain ${\wtilde C}^i_{k_i}(\CC)$ that lies in the fiber $\E_b$. Now, we know that $\{V_j\}$ is a geometric cover of $F_b$,  so ${\wtilde C}^i_{k_i}(\CC)$ lies  in some $V_j$.
    We conclude that $\CC$ is in $W_{i,j}$.
   
   (2) Each $W_{i,j}$ is a homotopy domain for $F$ and $G$.
    
    For the sake of simplicity we change the notations as follows: $U=U_i$, $\wtilde C=\wtilde C^i$, $k=k_i$, $V=V_j$.
    
     Let $K\colon V \times \mathcal{I}_l \rightarrow \E_b$ be a homotopy between $F\vert_V$ and $G\vert_V$ and let $\iota\colon \E_b \hookrightarrow \E$ be the inclusion of the fiber into the total category $\E$. The homotopy that we need is the  functor 
     $$H: W_{i,j} \times \mathcal{I}_{k+l+k} \rightarrow \E'$$
     given by
    $$H(c,n)= 
    \begin{cases}
             F\wtilde C(c,n)   &\text{if}\quad   0 \leq n \leq k \cr
             \iota K(\wtilde C_kc,n-k)  &\text{if} \quad k \leq n \leq k+l \cr
             G\wtilde C(c,k+l+k-n)  &\text{if}\quad  k+l \leq n \leq k+l+k\cr
\end{cases}.$$
   It only remains to check that   $H$ is well defined and that it is the wanted homotopy: 
   \begin{align*}
      H(c,0)=& F\wtilde C(c,0)=Fc \text{ since } \wtilde C_0 \text { is the inclusion},\\
      \iota K(\wtilde C_kc,0)=&\iota F(\wtilde C_kc)=F\wtilde C(c,k),\\
        \iota K(\wtilde C_kc,l)=&\iota G\wtilde C_kc=G\wtilde C(c,k),\\
       H(c,k+l+k)=&G\wtilde C(c,0)=Gc \text{ since } \wtilde C_0 \text{ is the inclusion}.
   \end{align*}
This proves that $\Di(F,G)\leq m+n$, as stated.
\end{proof}

\end{document}